\documentclass[12pt]{amsart}

\newtheorem{theorem}{Theorem}[section]
\newtheorem{lemma}[theorem]{Lemma}

\newtheorem{corollary}[theorem]{Corollary}

\theoremstyle{definition}

\newtheorem{example}[theorem]{Example}
\newtheorem{algorithm}[theorem]{Algorithm}
\newtheorem{remark}[theorem]{Remark}

\usepackage{amscd,amssymb}
\begin{document}

\title[Defining relations of low degree]
{Defining Relations of Low Degree of Invariants
of Two $4 \times 4$ Matrices}

\author[Vesselin Drensky and Roberto La Scala]
{Vesselin Drensky and Roberto La Scala}
\address{Institute of Mathematics and Informatics,
Bulgarian Academy of Sciences,
Acad. G. Bonchev Str., Block 8, 1113 Sofia, Bulgaria}
\email{drensky@math.bas.bg}

\address{Dipartimento di Matematica,
Universit\`a di Bari, Via E. Orabona 4,
70125 Bari, Italia}
\email{lascala@dm.uniba.it}

\thanks
{The research of the first author was partially supported by Grant
MI-1503/2005 of the Bulgarian National Science Fund.}
\thanks{The research of the second was partially supported
by Universit\`a di Bari and MIUR}


\begin{abstract}
The trace algebra $C_{nd}$ over a field of characteristic 0
is generated by all traces of products of
$d$ generic $n\times n$ matrices, $n,d\geq 2$.
Minimal sets of generators of $C_{nd}$ are known for $n=2$ and
3 for any $d$ and for $n=4$ and 5 and $d=2$. The explicit defining relations
between the generators are found for $n=2$ and any $d$ and for $n=3$, $d=2$ only.
Defining relations of minimal degree for $n=3$ and any $d$ are also known.
The minimal degree of the defining relations of any homogeneous
minimal generating set of $C_{42}$ is equal to 12. Starting with the generating set
given recently by Drensky and Sadikova, we have determined all relations of degree $\leq 14$.
For this purpose we have developed further algorithms
based on representation theory of the general linear
group and easy computer calculations with standard functions of Maple.
\end{abstract}

\maketitle

\section*{Introduction}

Let $K$ be any field of characteristic 0. All vector spaces, tensor products,
algebras considered in this paper are over $K$. Let
$X_i=\left(x_{pq}^{(i)}\right)$, $p,q=1,\ldots,n$, $i=1,\ldots,d$,
be $d$ generic $n\times n$ matrices. We consider the
pure (or commutative) trace algebra $C_{nd}$
generated by all traces of products
$\text{\rm tr}(X_{i_1}\cdots X_{i_k})$.
It coincides with the algebra of invariants of the general
linear group $GL_n=GL_n(K)$ acting by simultaneous conjugation on
$d$ matrices of size $n\times n$. The
algebra $C_{nd}$ is finitely generated.
An upper bound for the degree of the trace monomials sufficient to generate $C_{nd}$
is given in terms of the Nagata-Higman theorem in the theory of PI-algebras.
The defining relations of $C_{nd}$ are
described by the Razmyslov-Procesi theory \cite{R, P} in the language
of ideals of the group algebras of symmetric groups. For a background
on the algebras of matrix invariants see e.g. \cite{F, DF} and
for computational aspects of the theory see \cite{D2}.

Explicit minimal sets of generators of $C_{nd}$ are known
for $n=2$ and 3 for any $d$, and $n=4$ and 5 for $d=2$ only.
The exact upper bound of the degree $k\leq N(n)$ of the trace polynomials
$\text{\rm tr}(X_{i_1}\cdots X_{i_k})$ sufficient to generate $C_{nd}$ is
$N(2)=3$, $N(3)=6$, and $N(4)=10$.
Even less is known for the defining relations between these
minimal sets of generators. For details on the explicit form of the defining relations
for $n=2$, $d\geq 2$ see e.g. \cite{DF}.
For $n=3$, $d=2$, a minimal generating set of $C_{32}$ consisting of 11 trace monomials
of degree $\leq 6$ was found by Teranishi \cite{T1}.
He also calculated the Hilbert (or Poincar\'e) series of $C_{32}$.
It follows from his description that, with respect to these generators,
$C_{32}$ has a single defining relation of degree 12.
The explicit form of the relation was found by Nakamoto \cite{N}, over $\mathbb Z$,
with respect to a slightly different system of generators.
Abeasis and Pittaluga \cite{AP} found
a system of generators of $C_{3d}$, for any $d\geq 2$, in terms of
representation theory of the symmetric and general linear groups,
in the spirit of its usage in theory of PI-algebras.
Aslaksen, Drensky and Sadikova \cite{ADS}
gave the defining relation of $C_{32}$ with respect to the generators
from \cite{AP}. For $n=3$ and $d>2$ the defining relations of $C_{3d}$ seem to be
very complicated. Recently, Benanti and Drensky \cite{BD} have shown that
for all $d>2$ the minimal degree of the defining relations of $C_{3d}$
is equal to 7 and have found explicitly these relations with
respect to the generators from \cite{AP}.
For $d=3$ they have given also the relations of degree 8,
using additional information
from the Hilbert series of $C_{33}$ calculated by
Berele and Stembridge \cite{BS}.
Independently, the defining relations of the algebra $C_{33}$
have been studied in the recent master thesis of Hoge \cite{H}.
Using representation theory of general linear groups and computer calculations with Maple,
as in \cite{ADS} and \cite{BD}, he developed a general algorithm
and found the relations of degree 7 and some of the relations of degree 8.

For $C_{42}$, a set of generators was found by Teranishi \cite{T1, T2}
and a minimal set by Drensky and Sadikova \cite{DS}, in terms of the approach in \cite{AP}.
Djokovi\'{c} \cite{Dj} gave another minimal set of 32 generators
of $C_{42}$ consisting of trace monomials only (he found also a minimal set
of 173 generators of $C_{52}$).
Any homogeneous minimal generating set $\{u_i\mid i=1,\ldots,32\}$
of $C_{42}$ consists of $g_i$ elements of degree
$i=1,2,\ldots,10$, where
\begin{equation}\label{dimensions of generators}
\begin{array}{c}
g_1=2,\quad g_2=3,\quad g_3=4,\quad g_4=6,\quad g_5=2,\\
\\
g_6=4,\quad g_7=2,\quad g_8=4,\quad g_9=4,\quad g_{10}=1.\\
\end{array}
\end{equation}
Hence $C_{42}$ is isomorphic to the factor algebra $K[y_1,\ldots,y_{32}]/I$.
Defining $\text{\rm deg}(y_i)=\text{\rm deg}(u_i)$,
the ideal $I$ is homogeneous.
The comparison of the Hilbert series of $C_{42}$ calculated by
Teranishi \cite{T2} (with some typos) and corrected by
Berele and Stembridge \cite{BS},
with the Hilbert series of $K[y_1,\ldots,y_{32}]$ gives that
any homogeneous minimal system of generators of the ideal $I$ contains
no elements of degree $\leq 11$ and 5 elements of degree 12, see \cite{DS}.
The purpose of the present paper is to find the explicit form of the defining relations
of minimal degree for $C_{42}$, with respect to
the generating set in \cite{DS}. We have performed similar computations also for higher degrees,
up to 14.
The proofs are based on representation theory of $GL_2$
combined with computer calculations with Maple and develop further ideas
of \cite{ADS, DS}.
In particular, we have found a way to write the defining relations in a compact form.
Our methods are quite general and can be successfully
used for further investigation of generic trace algebras
and other algebras close to them.

Having in hand some defining relations of $C_{42}$, we face the problem what is their meaning.
We suggest the following point of view. It is known that the algebra $C_{nd}$ is Cohen-Macaulay.
It has a homogeneous system of parameters $u_1,\ldots,u_p$ which are algebraically independent
and $C_{nd}$ is a finitely generated free $K[u_1,\ldots,u_p]$-module. Here $p=(d-1)n^2+1$ is the
Krull dimension of $C_{nd}$. In our case the homogeneous system of parameters of $C_{42}$ consists
of 17 of the 32 generators $u_i$ of $C_{42}$, say $u_1,\ldots,u_{17}$,
and the free $K[u_1,\ldots,u_{17}]$-module $C_{42}$ is freely generated by a finite set of products
\begin{equation}\label{CM-algebra}
\{u_{18}^{a_{18}}\cdots u_{32}^{a_{32}}\mid (a_{18},\ldots a_{32})\in A\}
\end{equation}
for some set of indices $A$. The form of the relations of low degree which we have found
agrees with the fact that every product $u_{18}^{b_{18}}\cdots u_{32}^{b_{32}}$ can be written
as a linear combination of the elements from (\ref{CM-algebra}) with coefficients from
$K[u_1,\ldots,u_{17}]$ and gives some restrictions on the indices $(a_{18},\ldots a_{32})$.

\section{Preliminaries}

Till the end of the paper we fix $n=4$ and $d=2$ and denote by $X,Y$
the two generic $4\times 4$ matrices. We shall denote $C_{42}$ by $C$.
It is a standard trick to replace the generic matrices with generic traceless
matrices. We express $X$ and $Y$ in the form
\[
X=\frac{1}{4}\text{\rm tr}(X)e+x,\quad Y =\frac{1}{4}\text{\rm tr}(Y)e+y,
\]
where $e$ is the identity $4\times 4$ matrix and $x,y$ are
generic traceless matrices.
Then
\begin{equation}\label{replacing with traceless matrices}
C\cong K[\text{\rm tr}(X),{\rm tr}(Y)]\otimes C_0,
\end{equation}
where the algebra $C_0$ is generated by the traces of products
$\text{\rm tr}(z_1\cdots z_k)$, $z_i=x,y$, $2\leq k\leq 10$. Hence the problem for
the generators and the defining relations of $C$ can be replaced by a similar problem for
$C_0$.

As in the case of ``ordinary'' generic matrices, up to similarity
we may replace $x$ by a generic traceless diagonal matrix.
Although not essential, this results in a simplification from a computational
point of view. In fact, one of the worst drawback when computing with traces
of polynomials in generic matrices is that these are commutative polynomials
with a very high number of monomials. Then, without loss of generality we can
fix the two generic traceless matrices as
\[
x=\left( \begin{array}{cccc}
                x_{11} & 0 & 0 & 0\\
                0 & x_{22} & 0 & 0\\
                0 & 0 & x_{33} & 0 \\
                0 & 0 & 0 & -(x_{11}+x_{22}+x_{33})\\
             \end{array} \right),
\]
\[
y=\left( \begin{array}{cccc}
                y_{11} & y_{12} & y_{13} & y_{44} \\
                y_{21} & y_{22} & y_{23} & y_{24} \\
                y_{31} & y_{32} & y_{33} & y_{34} \\
                y_{41} & y_{42} & y_{43} & -(y_{11}+y_{22}+y_{33})\\
             \end{array} \right)
\]
We summarize now the necessary background on representation theory of general
linear groups $GL_d$. We shall state everything for $d = 2$ only.
See \cite{M, W} for details on polynomial representations of $GL_d$ and \cite{D1}
for their applications to PI-algebras.
The group $GL_2 = GL_2(K)$ acts in a canonical way on the vector space with
basis $\{x,y\}$ and this action induces a diagonal action on
the free associative algebra $K\langle x,y\rangle$:
\[
g(z_1\cdots z_k)=g(z_1)\cdots g(z_k),\quad z_i=x,y,\quad g\in GL_2.
\]
The action of $GL_2$ on $K\langle x,y\rangle$ induces an action on
the algebras $C$ and $C_0$. For $C_0$ it is given by
\[
g(\text{tr}(z_1\cdots z_k))=\text{tr}(g(z_1)\cdots g(z_k)),\quad
z_i=x,y,\quad g\in GL_2.
\]
The $GL_2$-module $K\langle x,y\rangle$ is a direct sum of
irreducible polynomial modules, described
in terms of partitions $\lambda=(\lambda_1,\lambda_2)$. We denote
by $W(\lambda)$ the corresponding $GL_2$-module.

The $GL_2$-submodules and factor modules $W$ of $K\langle x,y\rangle$
inherit its natural bigrading which counts the entries of $x$ and $y$
in each monomial. We denote by $W^{(p,q)}$ the corresponding homogeneous component
of degree $p$ and $q$ in $x$ and $y$, respectively.
The formal power series
\[
H(W,t,u)=\sum_{p,q\geq 0}\text{\rm dim}(W^{(p,q)})t^pu^q
\]
is called the Hilbert series of $W$. The Hilbert series of $W(\lambda)$ is the Schur
function $S_{\lambda}(t,u)$ which, in the case of two variables, has the simple form
\begin{equation}\label{definition of Schur function}
S_{\lambda}(t,u)=(tu)^{\lambda_2}(t^{\lambda_1-\lambda_2}+t^{\lambda_1-\lambda_2-1}u+
\cdots+tu^{\lambda_1-\lambda_2-1}+u^{\lambda_1-\lambda_2}).
\end{equation}
The Hilbert series of $W$ plays the role of its character. The module
$W(\lambda)$ participates in $W$ with multiplicity $m(\lambda)$, i.e.,
\[
W=\bigoplus (W(\lambda))^{\oplus m(\lambda)},\quad
m(\lambda)\in {\mathbb N}\cup\{0\},
\]
if and only if
\[
H(W,t,u)=\sum m(\lambda)S_{\lambda}(t,u).
\]

Let $C_0^+=\omega(C_0)$ be the augmentation ideal of $C_0$.
It consists of all trace polynomials $f(x,y)\in C_0$ without constant terms,
i.e., satisfying the condition $f(0,0) = 0$. Any minimal system of generators
of $C_0$ lying in $C_0^+$ forms a basis of the vector space $C_0^+$ modulo $(C_0^+)^2$.
Abeasis and Pittaluga \cite{AP} suggested to fix the minimal system of generators
of $C_{nd}$ in such a way that it spans a $GL_d$-module $G$. Then $C_{nd}$ is a homomorphic
image of the symmetric algebra $K[G]=\text{Sym}(G)$ and the defining relations correspond to the generators
of the kernel of the natural homomorphism $K[G]\to C_{nd}$.
Drensky and Sadikova \cite{DS} found that the minimal $GL_2$-module of generators of $C_{42}$
is decomposed as
\begin{equation}\label{decomposition of generators}
\begin{array}{c}
G=W(1,0)\oplus W(2,0)\oplus W(3,0)\oplus W(4,0)\oplus W(2,2)\\
\\
\oplus W(3,2)\oplus W(4,2)\oplus W(3,3)\oplus W(4,3)\\
\\
\oplus W(5,3)\oplus W(4,4)\oplus W(6,3)\oplus W(5,5).\\
\end{array}
\end{equation}
Hence the minimal generating $GL_2$-module $G_0$ of $C_0$ is the direct sum of those modules
in (\ref{decomposition of generators}) which are different from $W(1,0)$.
For $\lambda=(\lambda_1,\lambda_2)\not=(5,5)$, one may choose as a generator of
$W(\lambda_1,\lambda_2)\subset G_0$ the canonical element
\begin{equation}\label{hwv for general case}
w_{\lambda}(x,y)=\text{\rm tr}((xy-yx)^{\lambda_2}x^{\lambda_1-\lambda_2}).
\end{equation}
A generator of $W(5,5)$ may be chosen as
\begin{equation}\label{hwv for 55}
w_{(5,5)}(x,y)=\text{\rm tr}((xy-yx)^3(x^2y^2-xy^2x-yx^2y+y^2x^2)).
\end{equation}
In \cite{DS} it corresponds to the standard tableau
\[
\left[
\begin{array}{@{\hskip 3pt}c@{\hskip 3pt}c@{\hskip 3pt}c
              @{\hskip 3pt}c@{\hskip 3pt}c
              @{\hskip 3pt}}
1 & 3 & 5 & 7 & 8  \\
2 & 4 & 6 & 9 & 10 \\
\end{array}
\right].
\]
Since $C_0\cong K[G_0]/J$ for an ideal $J$ which is also graded,
the difference of the Hilbert series of $K[G_0]$ and $C_0$ gives the Hilbert series
of $J$. By \cite{DS}, the Hilbert series of $J$ is
\[
H(J,t,u)=H(C_0,t,u)-H(K[G_0],t,u)=
(S_{(7,5)}(t,u)+2S_{(6,6)}(t,u))
\]
\[
+(S_{(8,5)}(t,u)+2S_{(7,6)}(t,u))
+(2S_{(9,5)}(t,u)+6S_{(8,6)}(t,u)+2S_{(7,7)}(t,u))
\]
\[
+(2S_{(10,5)}(t,u)+9S_{(9,6)}(t,u)+7S_{(8,7)}(t,u))+\cdots.
\]
Hence, the $GL_2$-modules $R_{12}, R_{13}$, and $R_{14}$
of the defining relations of degree
12, 13, and 14 are, respectively,
\begin{equation}\label{decomposition of relations}
\begin{array}{c}
R_{12} = W(7,5)\oplus 2W(6,6),\\
\\
R_{13} = W(8,5)\oplus 2W(7,6),\\
\\
R_{14} = 2W(9,5)\oplus 6W(8,6)\oplus 2W(7,7).\\
\end{array}
\end{equation}

Any submodule $W(\lambda)=W(\lambda_1,\lambda_2)$ of $K\langle x,y\rangle$
is generated by a unique, up to a multiplicative constant,
homogeneous element $w_{\lambda}(x,y)$ of degree $\lambda_1$ and $\lambda_2$
in $x$ and $y$, respectively,
called the ``highest weight vector'' of $W(\lambda)$. It is characterized in
the following way, see \cite{DEP, ADF} and \cite{K} for the version which
we need. We state it for two variables only. Recall that a linear operator
$\delta$ on an algebra $R$ is called a derivation if $\delta(uv)=
\delta(u)v+u\delta(v)$ for all $u,v\in R$. We define a derivation
$\Delta$ of $K\langle x,y\rangle$ by putting
\begin{equation}\label{derivation for hwv}
\Delta(x)=0,\quad \Delta(y)=x
\end{equation}
and a linear operator $h\in GL_2$ by
$h(x)=x$, $h(y)=x+y$, i.e.,
\begin{equation}\label{triangular matrix for hwv}
h=\left(
\begin{matrix}
1 & 1\\
0 & 1\\
\end{matrix}
\right).
\end{equation}
The $GL_2$-submodules and factor modules of $K\langle x,y\rangle$
are invariant under the action of $\Delta$ and we can extend $\Delta$
also to tensor products, symmetric algebras, and other constructions with such modules. For example,
if $W_1,W_2\subset K\langle x,y\rangle$, we define $\Delta$ on the
tensor product $W_1\otimes W_2$ by $\Delta(w_1\otimes w_2) =
\Delta(w_1)\otimes w_2+w_1\otimes \Delta(w_2)$, $w_i\in W_i$.

\begin{lemma}\label{criterion for hwv}
{\rm (\cite{ADF, DEP, K}, see also \cite{BD})}
Let $\Delta$ and $h$ be defined as in {\rm (\ref{derivation for hwv})}
and {\rm (\ref{triangular matrix for hwv})}, respectively.
The homogeneous polynomial $w_{\lambda}(x,y) \in K\langle x,y\rangle$ of degree
$(\lambda_1,\lambda_2)$ is a highest weight vector for some $W(\lambda_1,\lambda_2)$
if and only if $\Delta(w_{\lambda}(x,y))=0$ or,
equivalently, $h(w_{\lambda}(x,y))=w_{\lambda}(x,y)$.
\end{lemma}

If $W_i\subset K\langle x,y\rangle$, $i=1,\ldots,k$,
are $k$ isomorphic copies of $W(\lambda)$ and $w_i\in W_i$ are highest weight
vectors, then $w_1,\ldots,w_k$ span a vector subspace
$V=Kw_1+\cdots +Kw_k$ of $K\langle x,y\rangle$
with the following property. The nonzero elements of $V$ are
highest weight vectors of submodules $W(\lambda)$
of the sum $W_1+\cdots+W_k$ and every highest weight vector can be obtained
in such a way. The sum $W_1+\cdots+W_k$ is direct if and only if
$w_1,\ldots,w_k$ are linearly independent.
The following statement is a direct consequence of Lemma \ref{criterion for hwv}.

\begin{corollary}\label{hwv as solutions of linear system}
If $W(\lambda)$, $\lambda=(\lambda_1,\lambda_2)$,
participates with multiplicity $m(\lambda)$ in the $GL_2$-submodule
$W$ of $K\langle x,y\rangle$, then the vector space of
the highest weight vectors $w_{\lambda}(x,y)$ is an $m(\lambda)$-dimensional
subspace of the homogeneous component $W^{(\lambda_1,\lambda_2)}$ of $W$.
Any basis $\{w_1,\ldots,w_{m(\lambda)}\}$ of this subspace generates
the direct sum $(W(\lambda))^{\oplus m(\lambda)}\subset W$ as $GL_2$-submodule.
\end{corollary}

\section{Algorithms}

For our concrete computations we need the explicit form of the highest weight
vectors in the symmetric algebra $K[G_0]$, where $G_0=G/W(1,0)$ generates $C_0$
and $G$ is given in (\ref{decomposition of generators}).
In \cite{ADS, BD, DS} a similar problem was solved by careful study of the
symmetric tensor powers $K[W(\lambda)]$ and their tensor products,
based on the Littlewood-Richardson rule (or, for $d=2$, on its partial case,
the Young rule) and symmetric tensor powers on the level of \cite{M, Th}.
In the present paper we use a simplified approach and work directly
in the symmetric algebra $K[G_0]$. (After we had finished the computations
we learned that a similar simplification was used independently by Hoge \cite{H}.)
We define the derivation $\Delta_1$ of $K\langle x,y\rangle$ by
\begin{equation}\label{derivation for generation of module}
\Delta_1(x)=y,\quad \Delta_1(y)=0
\end{equation}
and a linear operator $h_1\in GL_2$ by
$h_1(x)=x+y$, $h_1(y)=y$, i.e.,
\begin{equation}\label{triangular matrix for generation of module}
h_1=\left(\begin{matrix}
1&0\\
1&1\\
\end{matrix}\right).
\end{equation}
As in the case of $\Delta$ from (\ref{derivation for hwv})
we extend the action of $\Delta_1$ on $GL_2$-modules related with
$K\langle x,y\rangle$.
The following lemma gives an algorithm which finds a basis of
$W(\lambda)$.

\begin{lemma}\label{algorithm for basis of module}
If $\lambda=(a+b,b)$ and $w(x,y)\in W(\lambda)\subset K\langle x,y\rangle$ is
a highest weight vector, then the set
\begin{equation}\label{basis of module}
\left\{
w,\frac{\Delta_1(w)}{a},\frac{\Delta_1^2(w)}{a(a-1)},
\ldots,\frac{\Delta_1^a(w)}{a(a-1)\cdots 2\cdot 1}
\right\}
\end{equation}
is a basis of the module $W(\lambda)$.
Here $\Delta_1$ is the derivation defined in
{\rm (\ref{derivation for generation of module})}.
\end{lemma}

\begin{proof}
It is well known that, starting with a highest weight vector
$w\in W(\lambda)$, the homogeneous components of $h_1(w)$
form a basis of $W(\lambda)$, where $h_1\in GL_2$ is from
(\ref{triangular matrix for generation of module}).
Now the proof follows from the fact that, up to a multiplicative constant,
$\Delta_1^k(w)$ is equal to the homogeneous component of degree $(a+b-k,b+k)$
of $\varepsilon_1(w)$, where
\[
\varepsilon_1=\exp(\Delta_1)=1+\Delta_1/1!+\Delta_1^2/2!+\cdots
\]
is the related exponential automorphism of the
locally nilpotent derivation $\Delta_1$, and
$h_1=\exp(\Delta_1)$.
\end{proof}

\begin{example}\label{basis of form in two variables}
(i) The $GL_2$-module $F(a)$ of the forms of degree $a$ in the polynomial algebra
$K[x,y]$ in two variables $x,y$
is isomorphic to $W(a,0)$ and $w=x^a$ is its highest weight vector.
Since $\Delta_1(y)=0$, we obtain
\[
\Delta_1(w)=ax^{a-1}y,\Delta_1^2(w)=a(a-1)x^{a-2}y^2,\cdots,
\]
\[
\Delta_1^{a-1}(w)=a(a-1)\cdots 2\cdot xy^{a-1},
\Delta_1^a(w)=a(a-1)\cdots 2\cdot 1\cdot y^a,
\Delta_1^{a+1}(w)=0.
\]
Hence Lemma \ref{algorithm for basis of module} gives the basis of $F(a)$
\[
\{x^a,x^{a-1}y,\ldots,xy^{a-1},y^a\}.
\]

(ii) Consider the submodules of $G_0$ in (\ref{decomposition of generators}).
The basis of $W(4,0)$ consists of the highest weight vector
\[
w=\text{tr}(x^4),
\]
\[
\frac{\Delta_1(w)}{4}=\frac{1}{4}\text{tr}(yx^3+xyx^2+x^2yx+x^3y)=\text{tr}(x^3y),
\]
\[
\frac{\Delta_1^2(w)}{4\cdot 3}=\frac{1}{3}\text{tr}((yx^2+xyx+x^2y)y)
=\frac{1}{3}(2\text{tr}(x^2y^2)+\text{tr}(xyxy)),
\]
\[
\frac{\Delta_1^3(w)}{4\cdot 3\cdot 2}=\text{tr}(xy^3),
\]
\[
\frac{\Delta_1^4(w)}{4\cdot 3\cdot 2\cdot 1}=\text{tr}(y^4).
\]

The basis of $W(5,3)\subset G_0$ consists of
\[
w=\text{tr}((xy-yx)^3x^2),
\]
\[
\frac{\Delta_1(w)}{2}=\frac{1}{2}\text{tr}((xy-yx)^3(yx+xy)),
\]
\[
\frac{\Delta_1^2(w)}{2\cdot 1}=\text{tr}((xy-yx)^3y^2).
\]
Note that we make use of the fact that the trace of
a product does not change under a cyclic permutation of its factors.
\end{example}

Applying Corollary \ref{hwv as solutions of linear system}
we obtain the following algorithm which is in the base of our further computations.

\begin{algorithm}\label{algorithm for hwv} \hfill

{\it Input.}
A partition $\lambda=(\lambda_1,\lambda_2)$ and
a system of highest weight vectors $w_i\in W_i$, $i=1,\ldots,k$,
where each $W_i$ is an irreducible $GL_2$-submodule
of $K\langle x,y\rangle$.

{\it Output.} A basis of the vector space of highest weight vectors
$w_{\lambda}(x,y)$ in the symmetric algebra $K[W]$ of the direct sum
$W=W_1\oplus\cdots\oplus W_k$.

{\bf Step 1.} Applying Lemma \ref{algorithm for basis of module}, find
homogeneous bases $\{u_{i0},\ldots,u_{ia_i}\}$ of the modules $W_i$,
$i=1,\ldots,k$.

{\bf Step 2.} In $K[W]$, form all products
\begin{equation}\label{the products w}
w_p=\prod_{i=1}^k\prod_{j=0}^{a_i}u_{ij}^{r_{ij}},\quad
p=1,\ldots,P,
\end{equation}
\begin{equation}\label{the products v}
v_q=\prod_{i=1}^k\prod_{j=0}^{a_i}u_{ij}^{s_{ij}},\quad q=1,\ldots,Q,
\end{equation}
which are of degree $(\lambda_1,\lambda_2)$ and $(\lambda_1+1,\lambda_2-1)$,
respectively. Present each $\Delta(w_p)$ in the form
\[
\Delta(w_p)=\sum_{q=1}^Q\alpha_{qp}v_q,\quad  \alpha_{qp}\in K.
\]

{\bf Step 3.} Consider the element
\[
w=\sum_{p=1}^P \xi_pw_p,
\]
with unknown coefficients $\xi_p\in K$. Calculate
\[
\Delta(w)=\sum_{q=1}^Q\left(\sum_{p=1}^P\alpha_{qp}\xi_p\right)v_p.
\]

{\bf Step 4.} Solve the homogeneous linear system
\begin{equation}\label{homogeneous linear system}
\sum_{p=1}^P\alpha_{qp}\xi_p=0,\quad q=1,\ldots,Q,
\end{equation}
whose equations are obtained from the equation
$\Delta(w)=0$.

{\bf Step 5.} Any basis
\[
\{\Xi_r=(\xi_1^{(r)},\ldots,\xi_P^{(r)})\mid r=1,\ldots,s\}
\]
of the vector space of solutions of the system gives rise to
a basis of the space of highest weight vectors.
\end{algorithm}

\begin{remark}\label{simplification of computations}
Instead of solving one big system (\ref{homogeneous linear system}),
we may solve several systems of smaller size. Let $W_i=W(\nu^{(i)})$
for some partition $\nu^{(i)}$. For each
$m_1,\ldots,m_k$ such that
\[
\sum_{i=1}^km_i\vert\nu_i\vert=\vert\lambda\vert
\]
the vector space $V(m_1,\ldots,m_k)$
spanned on those elements $w_p$ from (\ref{the products w}) with
\[
\sum_{j=0}^{a_i}r_{ij}=m_i,\quad i=1,\ldots,k,
\]
is a $GL_2$-submodule of $K[W]$. Since
\[
K[W]=K[W_1\oplus\cdots\oplus W_k]\cong
K[W_1]\otimes\cdots\otimes K[W_k],
\]
we derive that
\begin{equation}\label{symmetrized powers}
V(m_1,\ldots,m_k)
\cong W_1^{\otimes_sm_1}\otimes\cdots\otimes W_k^{\otimes_sm_k},
\end{equation}
where $W_i^{\otimes_sm_i}$ is the $m_i$-th symmetrized tensor power
of $W_i$. The sum of all $V(m_1,\ldots,m_k)$ is direct, and we may
choose a basis of the vector space of the $\lambda$-highest weight vectors
in $K[W]$ as the union of the corresponding bases in $V(m_1,\ldots,m_k)$.
If $k>1$, the homogeneous linear systems corresponding to $V(m_1,\ldots,m_k)$
are simpler that the whole system (\ref{homogeneous linear system})
for most of the $\lambda$.
\end{remark}

Obvious modifications of Algorithm \ref{algorithm for hwv}
give the highest weight vectors in
other situations. For example, let $W_1$ and $W_2$ have homogeneous bases
$\{u_0,u_1,\ldots,u_p\}$ and $\{v_0,v_1,\ldots,v_q\}$, respectively.
If we want to find the highest weight vectors
in the tensor product $W_1\otimes W_2$,
we have to solve the homogeneous linear system
obtained from the equation
\[
\Delta\left(\sum\xi_{ij}u_i\otimes v_j\right)
=\sum \xi_{ij}(\Delta(u_i)\otimes v_j+u_i\otimes\Delta(v_j))=0,
\]
where the sum is on all $i,j$ such that $u_i\otimes v_j$ is homogeneous
of degree $(\lambda_1,\lambda_2)$.

\begin{remark}
If we want to find only the multiplicity of $W(\lambda)$ in
$K[W]$, where $W=W_1\oplus\cdots\oplus W_k$, we can proceed in the following way.
If $W_i=W(\nu^{(i)})$, then the Hilbert series of $W$ is a sum of Schur functions,
\[
H(W,t,u)=\sum_{i=1}^kS_{\nu^{(i)}}(t,u)=\sum a_{bc}t^bu^c,
\quad a_{bc}\in{\mathbb N}\cup \{0\}.
\]
Hence the Hilbert series of $K[W]$ is
\[
H(K[W],t,u)=\prod_{b,c}\frac{1}{(1-t^bu^c)^{a_{bc}}}=\sum_{p,q}h(p,q)t^pu^q,
\quad h(p,q)\in{\mathbb N}\cup \{0\}.
\]
By (\ref{definition of Schur function}) the Schur function $S_{\mu}(t,u)$
contains the summand $t^{\lambda_1}u^{\lambda_2}$ if and only if
$\mu_1+\mu_2=\lambda_1+\lambda_2$ and $\mu_1\geq\lambda_1$. This easily
implies that the multiplicity of $W(\lambda)$ is given by the formula
\begin{equation}\label{formula for multiplicities only}
m(\lambda)=h(\lambda_1,\lambda_2)-h(\lambda_1+1,\lambda_2-1).
\end{equation}
Similarly, if we want to find the multiplicity of $W(\lambda)$
in the tensor product $W_1\otimes\cdots \otimes W_k$, $W_i=W(\nu^{(i)})$,
$i=1,\ldots,k$, we have to present the product of the corresponding
Schur functions in the form
\[
\prod_{i=1}^kS_{\nu^{(i)}}(t,u)=\sum_{p,q} h(p,q)t^pu^q, \quad
h(p,q)\in{\mathbb N}\cup \{0\},
\]
and to obtain the multiplicity of $W(\lambda)$ by the formula
(\ref{formula for multiplicities only}).
\end{remark}

We want now to give a compact form for the highest weight vectors
of the tensor products $V(m_1,\ldots,m_k)$ defined in (\ref{symmetrized powers}).
We fix an order on the summands $W_i$ in the decomposition of $G_0=G/W(1,0)$
given in (\ref{decomposition of generators}). We put:
\begin{equation}\label{order of modules}
\begin{array}{c}
W_1= W(2,0),\quad W_2=W(3,0),\quad W_3=W(4,0),\\
\\
W_4=W(2,2),\quad W_5=W(3,2),\quad W_6=W(4,2),\\
\\
W_7=W(3,3),\quad W_8=W(4,3),\quad W_9=W(5,3),\\
\\
W_{10}=W(4,4),\quad W_{11}=W(6,3),\quad W_{12}=W(5,5).\\
\end{array}
\end{equation}
For each $W_i=W(\lambda)$ we fix a highest weight vector
$w_i=w_{\lambda}(x,y)$ given in (\ref{hwv for general case})
and (\ref{hwv for 55}).
Rewriting $\lambda=(\lambda_1,\lambda_2)$ in the form $\lambda=(a+b,b)$
we assume that $W_i=W(a_i+b_i,b_i)$.
The $GL_2$-module $W(a_i+b_i,b_i)$ is isomorphic to the tensor product
$\text{\rm det}^{b_i}\otimes W(a_i,0)$, where $\text{det}^{b_i}$
is the one-dimensional $GL_2$-module with $GL_2$-action defined by
\[
g(v)=(\det(g))^{b_i}\cdot v,\quad g\in GL_2, v\in \text{det}^{b_i}.
\]
The module $W(a_i,0)$ has a natural realization as
the module $F(a_i)$ of the forms of degree $a_i$ in two variables $x_i,y_i$.
We fix a nonzero element of $\text{det}^{b_i}$ and denote it by
$t_i^{b_i}$. Omitting the symbol $\otimes$ for the tensor product,
$\text{\rm det}^{b_i}\otimes F(a_i)$ has a basis
\[
\{t_i^{b_i}x_i^{a_i},t_i^{b_i}x_i^{a_i-1}y_i,
\ldots,t_i^{b_i}x_iy_i^{a_i-1},t_i^{b_i}y_i^{a_i}\}
\]
with action of $GL_2$ defined by
\[
g(t_i^{b_i}x_i^jy_i^{a_i-j})
=(\text{\rm det}(g))^{b_i}t_i^{b_i}(g(x_i))^j(g(y_i))^{a_i-j},\quad g\in GL_2.
\]
Using the highest weight vector $w_i$ of $W(a_i+b_i,b_i)$ from
(\ref{hwv for general case}) and (\ref{hwv for 55}), we fix the $GL_2$-module
isomorphism
\begin{equation}\label{isomorphism of modules}
\varphi_i:\text{\rm det}^{b_i}\otimes F(a_i)\to W_i=W(a_i+b_i,b_i)
\end{equation}
which sends $t_i^{b_i}x_i^{a_i}$ to $w_i(x,y)$.
The concrete form of the image of $t_i^{b_i}x_i^jy_i^{a_i-j}$ in $W(a_i+b_i,b_i)$
can be obtained applying
Lemma \ref{algorithm for basis of module}.
For the derivation $\Delta_1$ from (\ref{derivation for generation of module}),
\[
\Delta_1(\text{det}^{b_i})=0,
\]
\[
\frac{1}{j}\Delta_1(x_i^jy_i^{a_i-j})=x_i^{j-1}y_i^{a_i+1-j}
\]
and we define recursively
\[
\varphi_i(t_i^{b_i}x_i^{a_i})=w_i(x,y),
\]
\[
\varphi_i(t_i^{b_i}x_i^{j-1}y_i^{a_i+1-j})
=\frac{1}{j}\Delta_1(\varphi_i(t_i^{b_i}x_i^jy_i^{a_i-j})),\quad
j=a_i,a_i-1,\ldots,2,1.
\]
For example, if $\lambda=(6,3)$, then $W(\lambda)=W_{11}$
in the notation of (\ref{order of modules}),
\[
\varphi_{11}(t_{11}^3x_{11}^3)
=w_{11}(x,y)=\text{\rm tr}([x,y]^3x^3),
\]
\[
\varphi_{11}(t_{11}^3x_{11}^2y_{11})
=\frac{1}{3}\Delta_1(w_{11})
=\frac{1}{3}\text{\rm tr}([x,y]^3(yx^2+xyx+x^2y)),
\]
\[
\varphi_{11}(t_{11}^3x_{11}y_{11}^2)=\frac{1}{3}\text{\rm tr}([x,y]^3(y^2x+yxy+xy^2)),
\]
\[
\varphi_{11}(t_{11}^3y_{11}^3)=\text{\rm tr}([x,y]^3y^3).
\]
Now we extend the $GL_2$-module isomorphisms $\varphi_i$ to the symmetric algebras.
Let
\begin{equation}\label{isomorphism of symmetric algebras}
\Phi:K\left[\bigoplus_{i=1}^{12}\text{\rm det}^{b_i}\otimes F(a_i)\right]\to K[G_0]
\end{equation}
be defined by
\[
\Phi\left(\prod_{i=1}^{12}\prod_{j=0}^{a_i}(t_i^{b_i}x_i^jy_i^{a_i-j})^{c_{ij}} \right)
=\prod_{i=1}^{12}\prod_{j=0}^{a_i}(\varphi_i(t_i^{b_i}x_i^jy_i^{a_i-j}))^{c_{ij}},\quad c_{ij}\geq 0.
\]
In order to avoid the confusion and to distinguish e.g.
$(x_1y_1)^2 = (x_1y_1)\otimes (x_1y_1)$
and $(x_1^2)(y_1^2)=x_1^2\otimes y_1^2$ in
$F(2)^{\otimes_s2}$, in the summands where $\sum_{j=0}^{a_i}c_{ij}>1$,
we shall denote the elements $x_i^jy_i^{a_i-j}$ by $z_i^{(j,a_i-j)}$.
Hence, instead of $(x_1y_1)^2 = (x_1y_1)\otimes (x_1y_1)$
and $x_1^2y_1^2=x_1^2\otimes y_1^2$ we shall write
$(z_1^{(1,1)})^2$ and $(z_1^{(2,0)})(z_1^{(0,2)})$, respectively.
There is no confusion using $t_i$ because
\[
W(a_i+b_i,b_i)^{\otimes_sm_i}
\cong \text{\rm det}^{b_im_i}\otimes F(a_i)^{\otimes_sm_i}.
\]
For example, using the notation of (\ref{symmetrized powers})
and (\ref{order of modules}) one has
\[
V(1,2,1,0,\ldots,0)=W_1\otimes W_2^{\otimes_s2}\otimes W_3
\]
\[
= W(2,0)\otimes W(3,0)^{\otimes_s2}\otimes W(4,0),
\]
\[
\varphi_1(y_1^2)=\text{\rm tr}(y^2),
\]
\[
\varphi_2(x_2^3)=\varphi_2(z_2^{(3,0)})=\text{\rm tr}(x^3),\
\varphi_2(x_2y_2^2)=\varphi_2(z_2^{(1,2)})=\text{\rm tr}(xy^2),
\]
\[
\varphi_3(x_3^3y_3)=\text{\rm tr}(x^3y),
\]
\[
\Phi(y_1^2(z_2^{(3,0)})(z_2^{(1,2)})x_3^3y_3)
=\text{\rm tr}(y^2)\text{\rm tr}(x^3)\text{\rm tr}(xy^2)\text{\rm tr}(x^3y).
\]
For
\[
V(2,1,0,0,1,0,0,0,0,0,0,0)
=W_1^{\otimes_s2}\otimes W_2\otimes W_5
\]
\[
=W(2,0)^{\otimes_s2}\otimes W(3,0)\otimes W(3,2),
\]
\[
\varphi_1(x_1y_1)=\varphi_1(z_1^{(1,1)})=\text{\rm tr}(xy),
\]
\[
\varphi_1(y_1^2)=\varphi_1(z_1^{(0,2)})=\text{\rm tr}(y^2),
\]
\[
\varphi_2(x_2y_2^2)=\text{\rm tr}(xy^2),
\]
\[
\varphi_5(t_5^2x_5)=\text{\rm tr}([x,y]^2x),
\]
\[
\Phi((z_1^{(1,1)})(z_1^{(0,2)})x_2y_2^2t_5^2x_5)
=\text{\rm tr}(xy)\text{\rm tr}(y^2)\text{\rm tr}(xy^2)\text{\rm tr}([x,y]^2x).
\]
For
\[
V(2,0,0,2,0,\ldots,0)=W_1^{\otimes_s2}\otimes W_4^{\otimes_s2}
=W(2,0)^{\otimes_s2}\otimes W(2,2)^{\otimes_s2},
\]
\[
\varphi_1((z_1^{(2,0)})(z_1^{(0,2)}))
=\text{\rm tr}(x^2)\text{\rm tr}(y^2),
\]
\[
\varphi_4(t_4^4)=\text{\rm tr}^2([x,y]^2),
\]
\[
\Phi((z_1^{(2,0)})(z_1^{(0,2)})t_4^4)
=\text{\rm tr}(x^2)\text{\rm tr}(y^2)\text{\rm tr}^2([x,y]^2).
\]

\section{Computations and Results}

We shall explain now the computations for degree 12. From
(\ref{decomposition of relations}) we see that it is sufficient
to consider the cases $\lambda=(7,5)$ and $\lambda=(6,6)$ only.
First, we use Algorithm \ref{algorithm for hwv} to find the highest weight
vectors $w_{\lambda}(x,y)\in K[G_0]$. Applying Step 1 of the algorithm
we find bases of the submodules $W_1,\ldots,W_{12}$ of $G_0$. By Step 2,
we form all products (\ref{the products w}) and (\ref{the products v}) of degree
$(\lambda_1,\lambda_2)$ and $(\lambda_1+1,\lambda_2-1)$, respectively.

For $\lambda=(7,5)$ we obtain that $P=155$ and $Q=119$, i.e., there are
$155$ products (\ref{the products w}) of degree $(7,5)$
and $119$ products (\ref{the products v}) of degree $(8,4)$.
Applying Steps 3 and 4 we compute that the system
(\ref{homogeneous linear system}) has $s=36$ linearly independent solutions
which give rise to 36 linearly independent highest weight vectors.
Hence $W(7,5)$ participates with multiplicity 36 in $K[G_0]$.
We call these 36 highest weight vectors $w_1,\ldots,w_{36}$.

For $\lambda=(6,6)$ the corresponding data are
$P=185$, $Q=155$ and the number of the linear independent highest weight
vectors $w_{(6,6)}(x,y)\in K[G_0]$ is $s=30$.

The next step of the computations is to find the highest weight vectors of the
$GL_2$-modules $W(\lambda)\subset R_{12}$ of the defining relations of degree 12.
For $\lambda=(7,5)$ we proceed in the following way. We form the trace polynomial in $K[G_0]$
\begin{equation}\label{searching defining relations}
w=\sum_{i=1}^{36}\zeta_iw_i,
\end{equation}
where $w_i$ are the 36 linearly independent highest weight vectors corresponding to
the submodules $W(7,5)$ of $K[G_0]$ and $\zeta_i$ are unknown coefficients.
Then we evaluate $w$ on the generic traceless $4\times 4$ matrices $x$ and $y$
and obtain
\begin{equation}\label{evaluation on traceless matrices}
w(x,y)=\sum_{i=1}^{36}\zeta_iw_i(x,y)
=\sum_{p,q=1}^4\sum_{i=1}^{36}\zeta_iw_i^{(p,q)}(x,y)e_{pq},
\end{equation}
where the $(p,q)$-entry $w_i^{(p,q)}(x,y)$ of $w_i(x,y)$ is a homogeneous polynomial of degree 12
in the entries $x_{aa}, y_{b_1b_2}$ of $x$ and $y$. We require that $w(x,y)=0$ which is equivalent to
\[
w^{(p,q)}(x,y)=\sum_{i=1}^{36}\zeta_iw_i^{(p,q)}(x,y)=0,\quad p,q=1,2,3,4.
\]
We rewrite the relations $w^{(p,q)}(x,y)=0$ in the form
\[
w^{(p,q)}(x,y)=\sum_{c,d}\alpha_{cd}^{(p,q)}(\zeta_1,\ldots,\zeta_{36})
\prod_a x_{aa}^{c_a} \prod_{b_1,b_2}y_{b_1b_2}^{d_{b_1b_2}} = 0.
\]
Since the coefficients $\alpha_{cd}^{(p,q)}$ are equal to 0, we obtain a homogeneous linear system
\begin{equation}\label{system for relations}
\alpha_{cd}^{(p,q)}(\zeta_1,\ldots,\zeta_{36})=0
\end{equation}
with unknowns $\zeta_1,\ldots,\zeta_{36}$. The solutions of the system give rise to the highest weight vectors
which generate the submodules $W(7,5)$ of the $GL_2$-module $R_{12}$ of defining relations of degree 12.
The result of the computations is that the system has a unique nonzero solution which we shall give explicitly soon.

Similar computations for $\lambda=(6,6)$ give that the system, which corresponds to (\ref{system for relations})
in this case, has two linearly independent solutions. One of them is relatively simple:
\begin{equation}\label{simple relation in case 66}
v'_{(6,6)} = 3u_1(x,y) + 4u_2(x,y) + 6u_3(x,y) = 0,
\end{equation}
where
\[
\begin{array}{c}
u_1=\text{\rm tr}(x^2)\text{\rm tr}(y^3)\text{\rm tr}([x,y]^3x)
-\text{\rm tr}(y^2)\text{\rm tr}(xy^2)\text{\rm tr}([x,y]^3y)
\\ \\
-\ 2\text{\rm tr}(xy)\text{\rm tr}(xy^2)\text{\rm tr}([x,y]^3x)
+2\text{\rm tr}(xy)\text{\rm tr}(x^2y)\text{\rm tr}([x,y]^3y)
\\ \\
+\ \text{\rm tr}(y^2)\text{\rm tr}(x^2y)\text{\rm tr}([x,y]^3x)
-\text{\rm tr}(y^2)\text{\rm tr}(x^3)\text{\rm tr}([x,y]^3y),
\\ \\
u_2=-\text{\rm tr}(y^3)\text{\rm tr}([x,y]^3x^3)
+ 3\text{\rm tr}(xy^2)\text{\rm tr}([x,y]^3(yx^2+xyx+x^2y))
\\ \\
-\ 3\text{\rm tr}(x^2y)\text{\rm tr}([x,y]^3(y^2x+yxy+xy^2))
+\text{\rm tr}(x^3)\text{\rm tr}([x,y]^3y^3),
\\ \\
u_3=-\text{\rm tr}([x,y]^2x)\text{\rm tr}([x,y]^3y)
+\text{\rm tr}([x,y]^2y)\text{\rm tr}([x,y]^3x).
\\ \\
\end{array}
\]
Applying the isomorphism $\Phi$ from (\ref{isomorphism of symmetric algebras})
we rewrite $u_1,u_2,u_3$ as
\[
u_1=\Phi(-(x_1y_2 - y_1x_2)^2(x_2y_8 - y_2x_8)t_8^3).
\]
\[
u_2=\Phi((x_2y_{11} - y_2x_{11})^3t_{11}^3),
\]
\[
u_3=\Phi(-(x_5y_8-y_5x_8)t_5^2t_8^3),
\]
Hence (\ref{simple relation in case 66}) has the form
\begin{equation}\label{Phi-form of first case 66}
\begin{array}{c}
v'_{(6,6)} = \Phi(
-3 (x_1y_2 - y_1x_2)^2(x_2y_8 - y_2x_8)t_8^3
\\ \\
+ 4 (x_2y_{11} - y_2x_{11})^3t_{11}^3
- \ 6 (x_5y_8-y_5x_8)t_5^2t_8^3
) = 0.
\end{array}
\end{equation}
In the same notation the only solution for the case $\lambda=(7,5)$ is
\begin{equation}\label{Phi-form of case 75}
\begin{array}{c}
v_{(7,5)} = \Phi(
-\ 6 z_1^{(2,0)} (z_1^{(2,0)} z_1^{(0,2)} - (z_1^{(1,1)})^2)t_7^3
\\ \\
-\ 4((z_1^{(2,0)} z_1^{(0,2)} - 6 (z_1^{(1,1)})^2) x_9^2
            + 10 z_1^{(2,0)} z_1^{(1,1)}x_9 y_9  - 5(z_1^{(2,0)})^2 y_9^2 ) t_9^3
\\ \\
+\ 4 (x_1 y_2 - y_1x_2)x_2 ((x_1 y_2 + y_1x_2)x_8 - 2 x_1 x_2 y_8)t_8^3
\\ \\
+\ 16 (x_1 y_3 - y_1x_3)^2 x_3^2t_7^3
-\ x_1^2t_4^2 t_7^3
+ 8 x_1^2t_{12}^5
\\ \\
+ 28 (z_2^{(3,0)} z_2^{(1,2)} - (z_2^{(2,1)})^2)t_7^3
- 48  x_2 x_{11} (x_2 y_{11}- y_2x_{11})^2t_{11}^3
\\ \\
- 48 x_3^2 (x_3 y_9 - y_3x_9)^2t_9^3
- 16  x_9^2t_4^2 t_9^3
\\ \\
- 24 x_5 x_8 t_5^2 t_8^3
+ 4 x_6^2t_6^2  t_7^3
) = 0
\end{array}
\end{equation}
and the second relation for $\lambda=(6,6)$ is

\begin{equation}\label{Phi-form of second case 66}
\begin{array}{c}
v''_{(6,6)} = \Phi(
- 108 (z_1^{(2,0)} z_1^{(0,2)} - (z_1^{(1,1)})^2)^3
\\
+ 216 (z_1^{(2,0)} z_1^{(0,2)} - (z_1^{(1,1)})^2)
      (x_3^2 z_1^{(0,2)} - 2 x_3 y_3 z_1^{(1,1)}
      + y_3^2 z_1^{(2,0)})^2
\\
- 180 (z_1^{(2,0)} z_1^{(0,2)} - (z_1^{(1,1)})^2)^2t_4^2
\\
- 12 (  54 z_1^{(2,0)} (z_1^{(1,1)})^2 (z_2^{(1,2)})^2
      + 12 (z_1^{(2,0)})^2 z_1^{(0,2)} z_2^{(2,1)} z_2^{(0,3)}
\\
      + 30 z_1^{(2,0)} (z_1^{(1,1)})^2 z_2^{(2,1)} z_2^{(0,3)}
      - 42 (z_1^{(2,0)})^2 z_1^{(1,1)} z_2^{(1,2)} z_2^{(0,3)}
\\
      - 72 z_1^{(2,0)} z_1^{(1,1)} z_1^{(0,2)} z_2^{(2,1)} z_2^{(1,2)}
      + 9 (z_1^{(2,0)})^2 z_1^{(0,2)} (z_2^{(1,2)})^2
\\
      - 12 z_1^{(2,0)} z_1^{(1,1)} z_1^{(0,2)} z_2^{(3,0)} z_2^{(0,3)}
      - 42 z_1^{(1,1)} (z_1^{(0,2)})^2 z_2^{(3,0)} z_2^{(2,1)}
      \\
      + 54 (z_1^{(1,1)})^2 z_1^{(0,2)} (z_2^{(2,1)})^2
      - 54 (z_1^{(1,1)})^3 z_2^{(2,1)} z_2^{(1,2)}
\\
      + 30 (z_1^{(1,1)})^2 z_1^{(0,2)} z_2^{(3,0)} z_2^{(1,2)}
      + 7 (z_1^{(2,0)})^3 (z_2^{(0,3)})^2
\\
      + 7 (z_1^{(0,2)})^3 (z_2^{(3,0)})^2
      + 9 z_1^{(2,0)} (z_1^{(0,2)})^2 (z_2^{(2,1)})^2
\\
      - 2 (z_1^{(1,1)})^3 z_2^{(3,0)} z_2^{(0,3)}
      + 12 z_1^{(2,0)} (z_1^{(0,2)})^2 z_2^{(3,0)} z_2^{(1,2)})
\\
+ 216(z_1^{(2,0)} z_1^{(0,2)} - (z_1^{(1,1)})^2)
            (z_1^{(0,2)}x_6^2  - 2 z_1^{(1,1)}x_6 y_6  + z_1^{(2,0)}y_6^2) t_6^2
\\
+ 432 (z_1^{(0,2)}x_2^2
- 2 z_1^{(1,1)}x_2 y_2
\\
+ z_1^{(2,0)}y_2^2 )
            (- z_1^{(0,2)}x_2 x_5 + z_1^{(1,1)}(x_2 y_5 + x_5 y_2)
      - z_1^{(2,0)}y_2 y_5)t_5^2
\\
- 432 (
-2(z_1^{(2,0)})^2((z_3^{(1,3)})^2
- z_3^{(2,2)} z_3^{(0,4)})
\\
- 4 z_1^{(2,0)} z_1^{(1,1)}( z_3^{(3,1)} z_3^{(0,4)}
- z_3^{(2,2)} z_3^{(1,3)})
\\
- z_1^{(2,0)} z_1^{(0,2)} ((z_3^{(2,2)})^2
-z_3^{(4,0)} z_3^{(0,4)})
\\
(z_1^{(1,1)})^2(-5 (z_3^{(2,2)})^2
+ z_3^{(4,0)} z_3^{(0,4)}
 + 4 z_3^{(3,1)} z_3^{(1,3)})
\\
       - 4 z_1^{(1,1)} z_1^{(0,2)} (z_3^{(4,0)} z_3^{(1,3)}
      -z_3^{(3,1)} z_3^{(2,2)})
\\
+ 2 (z_1^{(0,2)})^2 (z_3^{(4,0)} z_3^{(2,2)}
- (z_3^{(3,1)})^2)
)
\\
+ 216 (z_1^{(0,2)}x_3^2  - 2 z_1^{(1,1)}x_3 y_3  + z_1^{(2,0)}y_3^2)^2t_4^2
\\
+ 33 (z_1^{(2,0)} z_1^{(0,2)}
            - (z_1^{(1,1)})^2)t_4^4
\\
+ 36 (- z_2^{(0,3)}x_3^3
      + 3 z_2^{(1,2)}x_3^2 y_3  - 3 z_2^{(2,1)}x_3 y_3^2
      + z_2^{(3,0)}y_3^3)
     (- x_1^2x_3z_2^{(0,3)}
\\
      + x_1(x_1 y_3 + 2 y_1 x_3) z_2^{(1,2)}
      - y_1(y_1x_3  + 2 x_1 y_3) z_2^{(2,1)} + y_1^2 y_3 z_2^{(3,0)})
\\
+ 45 (  x_1^2 (z_2^{(2,1)} z_2^{(0,3)} - (z_2^{(1,2)})^2)
\\
           - x_1 y_1 (z_2^{(3,0)} z_2^{(0,3)} - z_2^{(2,1)} z_2^{(1,2)})
           + y_1^2 (z_2^{(3,0)} z_2^{(1,2)} - (z_2^{(2,1)})^2))t_4^2
\\
- 108 (x_1 y_3 - y_1x_3)^2 (x_3 y_6 - y_3x_6)^2t_6^2
\\
+ 9 (x_1 y_6 - y_1x_6)^2t_4^2 t_6^2
- 108 (x_1 y_5 - y_1 x_5)^2t_5^4
\\
- 9 (  4 (z_2^{(2,1)})^3 z_2^{(0,3)}
     - 6 z_2^{(3,0)} z_2^{(1,2)} z_2^{(2,1)} z_2^{(0,3)}
     + (z_2^{(3,0)})^2 (z_2^{(0,3)})^2
\\
     + 4 z_2^{(3,0)} (z_2^{(1,2)})^3
     - 3 (z_2^{(2,1)})^2 (z_2^{(1,2)})^2)
\\
-144(x_2 y_3 - y_2x_3)^3 (x_3 y_5- y_3x_5) t_5^2
\\
- 108 ( (z_2^{(2,1)} z_2^{(0,3)} - (z_2^{(1,2)})^2) x_6^2
            - (z_2^{(3,0)} z_2^{(0,3)}
\\
            - z_2^{(2,1)} z_2^{(1,2)})x_6 y_6
            + (z_2^{(3,0)} z_2^{(1,2)} - (z_2^{(2,1)})^2)y_6^2  )t_6^2
\\
+ 432 (- (z_3^{(3,1)})^2 z_3^{(0,4)}
        + 2 z_3^{(3,1)} z_3^{(2,2)} z_3^{(1,3)}
\\
       - (z_3^{(2,2)})^3
       - z_3^{(4,0)} (z_3^{(1,3)})^2
       + z_3^{(4,0)} z_3^{(2,2)} z_3^{(0,4)})
\\
- 36(  z_3^{(4,0)} z_3^{(0,4)}
            - 4 z_3^{(1,3)} z_3^{(3,1)}
            + 3 (z_3^{(2,2)})^2) t_4^2
\\
+ 14 t_4^6
- 36 t_4^2 t_{10}^4
- 108 (z_6^{(2,0)} z_6^{(0,2)} - (z_6^{(1,1)})^2)t_6^4
+ 24 t_7^6
) = 0.
\end{array}
\end{equation}
We state the results as a theorem.

\begin{theorem}\label{computational results for degree 12}
The defining relations of degree $12$ of the pure trace algebra generated by two traceless
$4\times 4$ generic matrices form a $GL_2$-module isomorphic to $W(7,5)\oplus 2W(6,6)$.
The corresponding highest weight vectors
\[
v_{(7,5)}=0,\quad v'_{(6,6)}=0,\quad v''_{(6,6)}=0
\]
are given in {\rm (\ref{Phi-form of case 75})}, {\rm (\ref{Phi-form of first case 66})}, and
{\rm (\ref{Phi-form of second case 66})}.
\end{theorem}

The computational results for degree 13 and 14 are quite long expressions
and we shall discuss them in the next section.
Here for each $\lambda$ corresponding to a defining relation in
(\ref{decomposition of relations})
we give only the numbers $P$ and $Q$ of the vectors
(\ref{the products w}) and (\ref{the products v}),
the number $s$ of linearly independent
solutions of the system (\ref{homogeneous linear system}) and the multiplicity $r$ of $W(\lambda)$
from (\ref{decomposition of relations}):

\begin{equation}\label{results for relations of degree 13 and 14}
\begin{array}{c}
\lambda=(8,5): \quad P=203,\quad Q=136,\quad s=67,\quad r=1,\\
\\
\lambda=(7,6): \quad P=252,\quad Q=203,\quad s=49,\quad r=2,\\
\\
\lambda=(9,5): \quad P=284,\quad Q=188,\quad s=96,\quad r=2,\\
\\
\lambda=(8,6): \quad P=390,\quad Q=284,\quad s=106,\quad r=6,\\
\\
\lambda=(7,7): \quad P=418,\quad Q=390,\quad s=28,\quad r=2.\\
\end{array}
\end{equation}

\section{Conclusions}

The homogeneous system of parameters of $C_{42}$ found by
Teranishi \cite{T1, T2} contains all traces of degree $\leq 4$
and two elements of degree $(4,2)$ and $(2,4)$, respectively.
If in the system of Teranishi we remove $\text{tr}(X),\text{tr}(Y)$
and replace $X,Y$ with $x,y$, respectively, we obtain a homogeneous system of parameters
of the algebra $C_0$ generated by the generic traceless $4\times 4$ matrices $x$ and $y$.
Following \cite{DS}, we can choose for
a homogeneous system of parameters of $C_0$ any 13 trace polynomials which form a $K$-basis of
\[
W(2,0)\oplus W(3,0)\oplus W(4,0)\oplus W(2,2)=W_1\oplus W_2\oplus W_3\oplus W_4\subset G_0
\]
and two more trace polynomials
\[
\text{\rm tr}((xy-yx)^2x^2),\text{\rm tr}((xy-yx)^2y^2)\in W(4,2)=W_6\subset G_0.
\]
Hence in (\ref{CM-algebra}) we can choose for $u_{18},\ldots,u_{32}$ any basis of the direct sum
from (\ref{order of modules})
\[
W_5\oplus W_7\oplus W_8\oplus W_9\oplus W_{10}\oplus W_{11}\oplus W_{12}
\]
and the only homogeneous polynomial of degree $(3,3)$ in $W(4,2)=W_6$.
Using Lemma \ref{algorithm for basis of module} we construct
homogeneous bases $\{u_{i0},\ldots,u_{ia_i}\}$ of the modules $W_i=W(a_i+b_i,b_i)$ from
(\ref{order of modules}). In this notation, we fix
the homogeneous system of parameters of $C_0$ consisting of
\begin{equation}\label{homogeneous system of parameters}
\{u_{ij},u_{60},u_{62}\mid i=1,2,3,4,\ j=0,1,\ldots,a_i\}
\end{equation}
and complete it to a system of generators of $C_0$ by
\begin{equation}\label{other generators}
\{u_{ij},u_{61}\mid i=5,7,8,9,10,11,12, \ j=0,1,\ldots,a_i\}.
\end{equation}
It is easy to see that the finitely generated free $S$-module $C_0$, where
\begin{equation}\label{algebra S}
S=K[u_{ij}\mid i=1,2,3,4,\ j=0,1,\ldots,a_i,\ i=6,\ j=0,2],
\end{equation}
has a basis of the form
\begin{equation}\label{free generators as S-module}
B=\left\{\prod_{i,j} u_{ij}^{b_{ij}}\mid i=5,6,\ldots,12\right\},
\end{equation}
where the $u_{ij}$'s are from (\ref{other generators}) and
the $b_{ij}$'s belong to some set of indices. Hence every product of elements from
(\ref{other generators}) can be presented as a linear combination of elements in $B$ from
(\ref{free generators as S-module}) with coefficients in $S$. We shall use the defining relations
of degree 12,13, and 14, to give some restriction on the integers $b_{ij}$.

We start with the highest weight vectors $v_{(7,5)}, v'_{(6,6)}, v''_{(6,6)}$ from
(\ref{Phi-form of first case 66}), (\ref{Phi-form of case 75}), and
(\ref{Phi-form of second case 66}). The trace polynomial $v_{(7,5)}$ is of the form
\begin{equation}\label{first relation for free basis 75}
v_{(7,5)}=-24u_{50}u_{80}+4u_{60}u_{70}+\cdots,
\end{equation}
where $\cdots$ stays for the linear combination of products of the generators (\ref{other generators})
with coefficients which are polynomials in $S$, do not depend on $u_{60},u_{62}$, and are without constant term
(i.e., from the augmentation ideal $\omega(S)$ of $S$).
By Lemma \ref{algorithm for basis of module}, the $GL_2$-module generated by $v_{(7,5)}$
has a basis
\[
\left\{
v_{(7,5)},\frac{1}{2}\Delta_1(v_{(7,5)}),\frac{1}{2}\Delta_1^2(v_{(7,5)})
\right\}.
\]
Direct computations show that
\[
\begin{array}{c}
\Delta_1(v_{(7,5)})=-24(\Delta_1(u_{50})u_{80}+u_{50}\Delta_1(u_{80}))\\
\\
+4(\Delta_1(u_{60})u_{70}+u_{60}\Delta_1(u_{70})+\cdots,\\
\end{array}
\]
and, since $\Delta_1(u_{70})=0$,
\begin{equation}\label{second relation for free basis 75}
\frac{1}{2}\Delta_1(v_{(7,5)})=-24(u_{51}u_{80}+u_{50}u_{81})
+4u_{61}u_{70}+\cdots.
\end{equation}
Similarly
\begin{equation}\label{third relation for free basis 75}
\frac{1}{2}\Delta_1^2(v_{(7,5)})=-48u_{51}u_{81}
+4u_{62}u_{70}+\cdots.
\end{equation}
The equations (\ref{first relation for free basis 75}), (\ref{second relation for free basis 75}),
and (\ref{third relation for free basis 75}) imply that
\begin{equation}\label{relations for 75}
\begin{array}{c}
v_{(7,5)}\equiv -24u_{50}u_{80},\\
\\
\frac{1}{2}\Delta_1(v_{(7,5)})\equiv -24(u_{51}u_{80}+u_{50}u_{81})
+4u_{61}u_{70},\\
\\
\frac{1}{2}\Delta_1^2(v_{(7,5)})\equiv -48u_{51}u_{81}\\
\end{array}
\end{equation}
modulo $\omega(S)B$. In the same way, $v'_{(6,6)}, v''_{(6,6)}$ generate one-dimensional $GL_2$-modules
isomorphic to $W(6,6)$ and can be written in the form
\begin{equation}\label{relations for 66}
\begin{array}{c}
v'_{(6,6)}\equiv -6(u_{50}u_{81}-u_{51}u_{80}),\\
\\
v''_{(6,6)}\equiv 108u_{61}^2+24u_{70}^2\\
\end{array}
\end{equation}
modulo $\omega(S)B$. We order the trace polynomials from (\ref{other generators}) by
\begin{equation}\label{ordering of monomials}
u_{50}\succ u_{51}\succ u_{70}\succ u_{80}\succ u_{81}\succ u_{61}
\succ u_{90}\succ\cdots\succ u_{12,0},
\end{equation}
i.e., $u_{i_1j_1}\succ u_{i_2j_2}$ if $i_1<i_2$ or $i_1=i_2$, $j_1<j_2$, except the case
$u_{70}\succ u_{80}\succ u_{81}\succ u_{61}$.
Then we extend the order lexicographically on the products of (\ref{other generators}).
Hence, (\ref{relations for 75}) and (\ref{relations for 66}) give five relations such that,
modulo $\omega(S)B$, their leading monomials are
\begin{equation}\label{leading monomials of relations}
u_{50}u_{80},\quad u_{50}u_{81},\quad u_{51}u_{80},\quad u_{51}u_{81},\quad u_{70}^2.
\end{equation}
The defining relations of degree 13 and 14 which have been found in the same way as
the defining relations of degree 12 show that
the corresponding highest weight vectors are of the form
\begin{equation}\label{hwvs of degree 13}
\begin{array}{c}
v_{(8,5)}=u_{50}u_{90}-u_{60}u_{80}+\cdots,\\
\\
v'_{(7,6)}=4(u_{50}u_{91}-u_{51}u_{90})+(u_{60}u_{81}-u_{61}u_{80})+\cdots,\\
\\
v''_{(7,6)}=u_{50}u_{10,0}-2u_{70}u_{80}+\cdots\\
\end{array}
\end{equation}
\begin{equation}\label{hwvs of degree 14}
\begin{array}{c}
v'_{(9,5)}=2u_{50}u_{11,0}-u_{60}u_{90}+\cdots,\\
\\
v''_{(9,5)}=\cdots,\\
\\
v'_{(8,6)}=u_{60}u_{10,0}-2u_{70}u_{90}+\cdots,\\
\\
v''_{(8,6)}=u_{50}u_{11,1}-u_{51}u_{11,0}+u_{60}u_{91}-u_{61}u_{90}+\cdots,\\
\\
v'''_{(8,6)}=-7u_{60}u_{10,0}+12u_{70}u_{90}+12u_{80}^2+\cdots,\\
\\
v^{(4)}_{(8,6)}=\cdots,\quad v^{(5)}_{(8,6)}=\cdots,\quad v^{6}_{(8,6)}=\cdots,\\
\\
v'_{(7,7)}=-6(u_{60}u_{92}-2u_{61}u_{91}+u_{62}u_{90})
+u_{70}u_{10,0}+\cdots,\\
\\
v''_{(7,7)}=\cdots,\\
\end{array}
\end{equation}
with the same meaning of $\cdots$ as above. Applying several times the derivation $\Delta_1$
on the highest weight vectors from (\ref{hwvs of degree 13}) and (\ref{hwvs of degree 14})
we obtain that they generate irreducible $GL_2$-modules with bases which,
modulo $\omega(S)B$,
have leading monomials of the form
\begin{equation}\label{leading monomials of relations of degree 13}
\begin{array}{c}
u_{50}u_{90},\quad u_{50}u_{91},\quad u_{51}u_{90},\quad u_{50}u_{92},\quad u_{51}u_{91},
\quad u_{51}u_{92},\\
\\
u_{50}u_{10,0}, \quad u_{51}u_{10,0}.\\
\end{array}
\end{equation}
The leading monomials in the case of degree 14 are
\begin{equation}\label{leading monomials of relations of degree 14}
\begin{array}{c}
u_{50}u_{11,0},\quad u_{50}u_{11,1},\quad u_{51}u_{11,0},\quad
u_{50}u_{11,2},\\
\\
u_{51}u_{11,1},\quad u_{50}u_{11,3},
\quad u_{51}u_{11,2},\quad u_{51}u_{11,3},\\
\\
u_{70}u_{90},\quad u_{70}u_{91},\quad u_{70}u_{92},\\
\\
u_{70}u_{10,0},\quad u_{80}^2,\quad u_{80}u_{81},\quad u_{81}^2.\\
\end{array}
\end{equation}

\begin{theorem}
Let us fix the homogeneous system of parameters
{\rm (\ref{homogeneous system of parameters})} of $C_0$,
complete it to a system of generators by
{\rm (\ref{other generators})}, and let $S$ be defined in
{\rm (\ref{algebra S})}.
The finitely generated free $S$-module $C_0$
has a basis of the form
{\rm (\ref{free generators as S-module})} such
the products in $B$ do not contain factors
$u_{i_1j_1}u_{i_2j_2}$ from the lists given in
{\rm (\ref{leading monomials of relations}),
(\ref{leading monomials of relations of degree 13}),
(\ref{leading monomials of relations of degree 14})}.
\end{theorem}

\begin{proof}
We order the elements (\ref{other generators}) by
(\ref{ordering of monomials}). The leading monomials
of the defining relations of degree 12, 13, and 14 are given
in (\ref{leading monomials of relations}),
(\ref{leading monomials of relations of degree 13}),
and (\ref{leading monomials of relations of degree 14}).
If the free generating set contains a monomial from these lists,
we can replace it by a linear combination
of monomials which are lower in the lexicographic order and
monomials from $\omega(S)C_0$.
\end{proof}

\begin{remark}
The generating function of the leading monomials from
(\ref{leading monomials of relations}),
(\ref{leading monomials of relations of degree 13}),
and (\ref{leading monomials of relations of degree 14})
is equal to $L(t,u)=L_{12}+L_{13}+L_{14}$, where
\[
L_{12}(t,u)=t^7u^5+3t^6u^6+t^5u^7=S_{(7,5)}(t,u)+2S_{(6,6)}(t,u),
\]
\[
L_{13}(t,u)=t^8u^5+3t^7u^6+3t^6u^7+t^5u^8=S_{(8,5)}(t,u)+2S_{(7,6)}(t,u),
\]
\[
L_{14}(t,u)=t^9u^5+4t^8u^6+5t^7u^7+4t^6u^8+t^5u^9
\]
\[
=S_{(9,5)}(t,u)+3S_{(8,6)}(t,u)+S_{(7,7)}(t,u).
\]
Comparing with the Hilbert series $H(R_i,t,u)$
of the defining relations $R_i$ of degree $i=12, 13, 14$,
from (\ref{decomposition of relations}), respectively,
we see that
$L_{12}(t,u)=H(R_{12},t,u)$, $L_{13}(t,u)=H(R_{13},t,u)$, and
\[
H(R_{14},t,u)-L_{14}(t,u)=
S_{(9,5)}(t,u)+ 3S_{(8,6)}(t,u)+ S_{(7,7)}(t,u).
\]
The explanation is the following. We multiply the trace polynomials
\[
v_{(7,5)},\quad \frac{1}{2}\Delta_1(v_{(7,5)}), \quad \frac{1}{2}\Delta_1^2(v_{(7,5)}),
\quad v'_{(6,6)},\quad v''_{(6,6)}
\]
 of degree 12 by the polynomials $u_{10},u_{11}\in W_1=W(2,0)\subset G_0$
and obtain linearly independent relations of degree 14 with generating
function which turns to be equal to the difference
$H(R_{14},t,u)-L_{14}(t,u)$.
Hence the new relations of degree 14, which cannot be obtained
from relations of lower degree, form a $GL_2$-module isomorphic to
\[
W(9,5)\oplus 3W(8,6)\oplus W(7,7).
\]
\end{remark}

\section*{Acknowledgements}

This project was started when the first author visited the
University of Bari. He is very grateful for the warm hospitality
and the creative atmosphere during his stay there.

\end{document}